\documentclass[11pt,reqno,tbtags,draft]{amsart}
\usepackage{amssymb}
\usepackage{url}
\usepackage[square,numbers]{natbib}
\bibpunct[, ]{[}{]}{;}{n}{,}{,}
\title[Tail bounds for geometric and exponential variables]
{Tail bounds for sums of geometric and exponential variables}

\date{28 June, 2014; typo corrected 24 September, 2017}
\author{Svante Janson}
\thanks{Partly supported by the Knut and Alice Wallenberg Foundation}
\address{Department of Mathematics, Uppsala University, PO Box 480,
SE-751~06 Uppsala, Sweden}
\email{svante.janson@math.uu.se}
\newcommand\urladdrx[1]{{\urladdr{\def~{{\tiny$\sim$}}#1}}}
\urladdrx{http://www2.math.uu.se/~svante/}

\subjclass[2010]{} 
\numberwithin{equation}{section}

\renewcommand\le{\leqslant}
\renewcommand\ge{\geqslant}

\allowdisplaybreaks

\newtheorem{theorem}{Theorem}[section]
\newtheorem{lemma}[theorem]{Lemma}

\newtheorem{corollary}[theorem]{Corollary}

\theoremstyle{definition}

\newtheorem{remark}[theorem]{Remark}

\newtheorem*{ack}{Acknowledgement}

\theoremstyle{remark}

\newenvironment{romenumerate}[1][0pt]{% optional argument changes indentation
\addtolength{\leftmargini}{#1}\begin{enumerate}% gives (i), (ii) etc.
 }{\end{enumerate}}

\newcounter{oldenumi}
{\setcounter{oldenumi}{\value{enumi}}
\begin{romenumerate} \setcounter{enumi}{\value{oldenumi}}}
{\end{romenumerate}}

\newcounter{thmenumerate}
\newenvironment{thmenumerate}
{\setcounter{thmenumerate}{0}%
 \def\item{\par% \ifnum\thethmenumerate=0\else\par\fi %\noindent\fi
 \refstepcounter{thmenumerate}\textup{(\roman{thmenumerate})\enspace}}
}
{}

\newcounter{romxenumerate}   %less indented than standard.

\newcounter{xenumerate}   %no left indentation; thus wider lines

\newcommand{\refT}[1]{Theorem~\ref{#1}}
\newcommand{\refC}[1]{Corollary~\ref{#1}}
\newcommand{\refL}[1]{Lemma~\ref{#1}}

\newcommand\marginal[1]{\marginpar[\raggedleft\tiny #1]{\raggedright\tiny#1}}
\newcommand\REM[1]{{\raggedright\texttt{[#1]}\par\marginal{XXX}}}

\newenvironment{comment}{\setbox0=\vbox\bgroup}{\egroup} %deletes!

\begingroup
  \divide\count255 by 60
  \multiply\count255 by -60
  \advance\count255 by \time
  \ifnum \count255 < 10 \xdef\klockan{\the\count1.0\the\count255}
  \else\xdef\klockan{\the\count1.\the\count255}\fi
\endgroup

\newcommand{\sumki}{\sum_{k=1}^\infty}

\newcommand{\sumin}{\sum_{i=1}^n}

\newcommand{\prodin}{\prod_{i=1}^n}

\newcommand\xpar[1]{(#1)}
\newcommand\bigpar[1]{\bigl(#1\bigr)}
\newcommand\Bigpar[1]{\Bigl(#1\Bigr)}
\newcommand\biggpar[1]{\biggl(#1\biggr)}

\newcommand\xcpar[1]{\{#1\}}

\def\rompar(#1){\textup(#1\textup)}    % usage: \rompar(...)

\newcommand\xqfrac[2]{#1/(#2)}

\def\xexp(#1){e^{#1}}
\newcommand\ceil[1]{\lceil#1\rceil}

\newcommand\Ntoo{\ensuremath{{N\to\infty}}}

\newcommand\punkt{.\spacefactor=1000}    % om problem!
    
\newcommand\ie{i.e\punkt}
\newcommand\eg{e.g\punkt}

\newcommand{\as}{a.s\punkt}

\newcommand{\tend}{\longrightarrow}
\newcommand\dto{\overset{\mathrm{d}}{\tend}}

\newcounter{CC}
\newcounter{cc}
\newcommand\E{\operatorname{\mathbb E{}}}
\renewcommand\P{\operatorname{\mathbb P{}}}

\newcommand\Var{\operatorname{Var}}

\newcommand\Exp{\operatorname{Exp}}

\newcommand\Ge{\operatorname{Ge}}

\newcommand\gl{\lambda}

\newcommand\gss{\sigma^2}
\newcommand\eps{\varepsilon}

\renewcommand\phi{\xxx}  %% WARNING

\newcommand\ett[1]{\boldsymbol1\xcpar{#1}}

\newcommand\qw{^{-1}}
\newcommand\qww{^{-2}}

\newcommand\qqq{^{1/3}}

\newcommand{\pgf}{probability generating function}

\newcommand\px{p_*}
\newcommand\ax{a_*}
\newcommand\lex{\gl\mu}
\newcommand\NN{^{(N)}}
\begin{document}

\begin{comment}  % Some suggestions:

60 Probability theory and stochastic processes
60C Combinatorial probability
60C05 Combinatorial probability

60F Limit theorems [See also 28Dxx, 60B12]
60F05 Central limit and other weak theorems
60F17 Functional limit theorems; invariance principles

60G Stochastic processes
60G09 Exchangeability
60G55 Point processes

\end{comment}

\begin{abstract} 
We give explicit bounds for the tail probabilities for sums of independent
geometric or exponential variables, possibly with different parameters.
\end{abstract}

\maketitle

\section{Introduction and notation}\label{S:intro}

Let $X=\sumin X_i$, where $n\ge1$ and $X_i$, $i=1,\dots,n$, 
are independent geometric random variables
with possibly different distributions: %; more precisely 
$X_i\sim \Ge(p_i)$
with $0<p_i\le 1$, \ie,
\begin{equation}\label{geo}
  \P(X_i=k) = p_i(1-p_i)^{k-1},\qquad k=1,2,\dots.
\end{equation}
Our goal is to estimate the tail probabilities $\P(X\ge x)$.
(Since $X$ is integer-valued, it suffices to consider integer $x$. However,
it is convenient to allow arbitrary real $x$, and we  do so.)

We define
\begin{align}
  \mu&:=\E X = \sumin \E X_i = \sumin \frac{1}{p_i},
\\
\px&:=\min_i p_i.
\end{align}
We shall see that $\px$ plays an important role in our estimates, which
roughly speaking show that the tail probabilities of $X$ decrease at 
about the same rate as the tail probabilities of $\Ge(\px)$, \ie, as 
for the variable $X_i$ 
with smallest $p_i$ and thus fattest tail.

Recall the simple and well-known fact that \eqref{geo} implies that,
  for any non-zero $z$ such that $|z|(1-p_i)<1$, 
  \begin{equation}\label{pgf}
	\E z^{X_i} = \sumki z^k \P(X_i=k)=\frac{p_i z}{1-(1-p_i)z}
=\frac{p_i }{z\qw-1+p_i}.
  \end{equation}

For future use, note that 
since $x\mapsto-\ln(1-x)$ is convex on $(0,1)$ and $0$ for $x=0$, 
\begin{equation}\label{tp}
  -\ln(1-x) \le -\frac{x}y\ln(1-y),
\qquad 0< x\le y<1.
\end{equation}

\begin{remark}
  The theorems and corollaries below
hold also, with the same proofs, for infinite sums
  $X=\sum_{i=1}^\infty X_i$, provided $\E X=\sum_i p_i\qw<\infty$.
\end{remark}

\begin{ack}
  This work was initiated during
the 25th International Conference on Probabilistic, Combinatorial and
Asymptotic Methods for the Analysis of Algorithms, AofA’14, in
Paris-Jussieu, June 2014, in response to a question by Donald Knuth.
%see \cite{??}.
I thank Donald Knuth and Colin McDiarmid for helpful discussions.
\end{ack}

\section{Upper bounds for the upper tail}

We begin with a simple upper bound obtained by  
the classical method of estimating the moment generating function
(or \pgf) and
using the standard inequality (an instance of Markov's inequality)
\begin{equation}
  \label{markov}
\P(X\ge x)\le z^{-x}\E z^X,
\qquad z\ge1,
\end{equation}
or equivalently
\begin{equation}
  \label{markov2}
\P(X\ge x)\le e^{-tx}\E e^{tX},
\qquad t\ge0.
\end{equation}
(Cf.~the related ``Chernoff bounds'' for the binomial distribution that are
proved by this method, see \eg{} \cite[Theorem 2.1]{JLR}, and
see \eg{} \cite{BLM} for other applications of this method.
See also \eg{} \cite[Chapter 2]{DemboZeitouni} or \cite[Chapter 27]{Kallenberg}
for more general large deviation theory.)

\begin{theorem}\label{T1}
  For any $p_1,\dots,p_n\in(0,1]$ and any $\gl\ge1$,
\begin{equation}\label{xp}
  \P(X\ge \lex)
\le 
 e^{-\px\mu(\gl-1-\ln \gl)}.
\end{equation}
\end{theorem}

\begin{proof}
If $0\le t<p_i$, then   $e^{-t}-1+p_i\ge p_i-t>0$, and thus by \eqref{pgf},
\begin{equation}
  \E e^{t X_i} 
%= \frac{p_i e^t}{1-(1-p_i)e^t}
= \frac{p_i}{e^{-t}-1+p_i}
\le  \frac{p_i}{p_i-t}
=\Bigpar{1-\frac{t}{p_i}}\qw.
\end{equation}
Hence, if $0\le t <\px=\min_i p_i$, then
\begin{equation}
  \E e^{t X} = \prodin \E e^{t X_i} 
\le \prodin \Bigpar{1-\frac{t}{p_i}}\qw
\end{equation}
and, by \eqref{markov2},
\begin{equation}\label{bo}
  \begin{split}
\P(X\ge \gl \mu)\le e^{-t\gl\mu} \E e^{t X}
\le \exp \biggpar{-t\gl\mu + \sumin-\ln \Bigpar{1-\frac{t}{p_i}}}.
  \end{split}
\end{equation}
By \eqref{tp} and  $0<\px/p_i\le1$,
we have, for $0\le t<\px$,
\begin{equation}\label{tpp}
-\ln \Bigpar{1-\frac{t}{p_i}}
 \le-\frac{\px}{p_i}\ln\Bigpar{1-\frac{t}{\px}}.
\end{equation}
Consequently, 
\eqref{bo} yields
\begin{equation}\label{box}
  \begin{split}
\P(X\ge \gl \mu)
&\le \exp\biggpar{-t\gl\mu -\ln \Bigpar{1-\frac{t}{\px}} \sumin
  \frac{\px}{p_i}}
\\&
= \exp\biggpar{-t\gl\mu -\px\mu\ln \Bigpar{1-\frac{t}{\px}}}.
  \end{split}
\end{equation}
%Choose $t=(1-\eps)\px$ for some $\eps\in(0,1)$.
Choosing $t=(1-\gl\qw)\px$ (which is optimal in \eqref{box}), we obtain
\eqref{xp}. 
\end{proof}

As a corollary we obtain a bound that is generally much cruder, but has the
advantage of not depending on the $p_i$'s at all.

\begin{corollary}\label{C1}
    For any $p_1,\dots,p_n\in(0,1]$ and any $\gl\ge1$,
\begin{equation}\label{c1}
  \P(X\ge \lex)
\le \gl e^{1-\gl} = e\gl e^{-\gl}.
\end{equation}
\end{corollary}

\begin{proof}
  Use $\mu\ge1/p_i$ for each $i$, and thus $\mu\px\ge1$ in \eqref{xp}.
(Alternatively, use $t=(1-\gl\qw)/\mu$ in \eqref{box}.)
\end{proof}

The bound in \refT{T1}
is rather sharp in many cases.
Also the cruder \eqref{c1} is almost 
sharp for $n=1$ (a single $X_i$) and small
$\px=p_1$; in this case
$\mu=1/p_1$ and
\begin{equation}
\P(X\ge\lex)=(1-p_1)^{\ceil{\gl\mu}-1}
=\exp\bigpar{\gl+O(\gl p_1)}  .
\end{equation}
Nevertheless,
we can improve \eqref{xp} somewhat, in particular when $\px=\min_i p_i$ is
not small, 
by using more careful estimates.

\begin{theorem}\label{T2}
For any $p_1,\dots,p_n\in(0,1]$ and any $\gl\ge1$,
  \begin{equation}\label{t2}
	\P(X\ge \lex)
\le\gl\qw (1-\px)^{(\gl-1-\ln\gl)\mu}.
  \end{equation}
\end{theorem}

The proof is given below. We note that \refT{T2} implies a minor improvement
of \refC{C1}:
\begin{corollary}\label{C2}
    For any $p_1,\dots,p_n\in(0,1]$ and any $\gl\ge1$,
\begin{equation}\label{c2}
  \P(X\ge \lex)
\le  e^{1-\gl} .
\end{equation}
\end{corollary}

\begin{proof}
  Use \eqref{t2} and $(1-\px)^\mu\le e^{-\px\mu}\le e^{-1}$.
\end{proof}

We begin the proof of \refT{T2}
with two lemmas yielding a minor improvement
of \eqref{markov}
using the fact that the variables are
geometric. (The lemmas actually use only that one of the variables is
geometric.) 

\begin{lemma}
  \label{L0}
  \begin{thmenumerate}
  \item 
For any integers $j$ and $k$ with $j\ge k$,
\begin{equation}\label{ao}
  \P(X\ge j)\ge(1-\px)^{j-k}  \P(X\ge k).
\end{equation}
\item 
For any real numbers $x$ and $y$ with $x\ge y$,
\begin{equation}\label{aoxy}
  \P(X\ge x)\ge(1-\px)^{x-y+1}  \P(X\ge y).
\end{equation}
  \end{thmenumerate}
\end{lemma}

\begin{proof}
(i).
  We may without loss of generality assume that $\px=p_1$.
Then, for any integers $i,j,k$ with $j\ge k$,
\begin{equation}
  \P(X\ge j\mid X-X_1=i)
=\P(X_1\ge j-i)
=(1-\px)^{(j-i-1)_+},
\end{equation}
and similarly for $  \P(X\ge k\mid X-X_1=i)$.
Since $(j-i-1)_+\le j-k+(k-i-1)_+$, it follows that
\begin{equation}
  \P(X\ge j\mid X-X_1=i)
\ge(1-\px)^{j-k}  \P(X\ge k\mid X-X_1=i)
\end{equation}
for every $i$, and thus \eqref{ao} follows by taking the expectation.

(ii). For real $x$ and $y$ we obtain from \eqref{ao} 
\begin{equation}%\label{aoxy}
  \begin{split}
  \P(X\ge x)&=\P(X\ge\ceil x)\ge(1-\px)^{\ceil x-\ceil y}  \P(X\ge \ceil y)
\\&
\ge(1-\px)^{x- y+1}  \P(X\ge  y).	
  \end{split}
\end{equation}
  \end{proof}

\begin{lemma}\label{L1}
  For any $x\ge0$ and $z\ge1$ with $z(1-\px)<1$,
  \begin{equation}\label{l1}
	\P(X\ge x)\le \frac{1-z(1-\px)}{\px}z^{-x}\E z^X.
  \end{equation}
\end{lemma}
\begin{proof}
Since $z\ge1$, \eqref{ao} implies that for every $k\ge1$,
\begin{equation}
  \begin{split}
\E z^X
&\ge
\E(z^X\cdot\ett{X\ge k})
=\E\biggpar{\biggpar{z^k+(z-1)\sum_{j=k}^{X-1}z^j}\ett{X\ge k}}	
\\&
=\E\biggpar{z^k\ett{X\ge k}+(z-1)\sum_{j=k}^{\infty}z^j\ett{X\ge j+1}}	
\\&
=z^k\P\xpar{X\ge k}+(z-1)\sum_{j=k}^{\infty}z^j\P\xpar{X\ge j+1}
\\&
\ge z^k\P\xpar{X\ge k}\biggpar{1+(z-1)\sum_{j=k}^{\infty}z^{j-k}(1-\px)^{j+1-k}}
\\&
=
 z^k\P\xpar{X\ge k}\biggpar{1+\frac{(z-1)(1-\px)}{1-z(1-\px)}}
\\&
=
 z^k\P\xpar{X\ge k}\frac{\px}{1-z(1-\px)}.
  \end{split}
\end{equation}
The result \eqref{l1}
follows when $x=k$ is a positive integer. The general case
follows by taking $k=\max(\ceil x,1)$ since then $\P(X\ge x)=\P(X\ge k)$.
\end{proof}

\begin{proof}[Proof of \refT{T2}]
We may assume that $\px<1$. (Otherwise every $p_i=1$ and $X_i=1$ a.s., so
$X=n=\mu$ \as{} and the result is trivial.)
We then choose 
\begin{equation}
  \label{z}
z:=\frac{\gl-\px}{\gl(1-\px)},
\end{equation}
\ie,
\begin{equation}\label{zqw}
  z\qw = \frac{\gl(1-\px)}{\gl-\px}
= 1-\frac{(\gl-1)\px}{\gl-\px}
;
\end{equation}
note that $z\qw\le1$ so $z\ge1$
and $z\qw>1-\px\ge1-p_i$ for every $i$.
Thus, by \eqref{pgf},
\begin{equation}\label{qk}
  \E z^X = \prodin \E z^{X_i} 
=\prodin\frac{p_i}{z\qw-1+p_i}
=\prodin\frac{1}{1-(1-z\qw)/p_i}
.
\end{equation}
 By \eqref{qk}, \eqref{tpp} (with $t=1-z\qw<\px$) and \eqref{zqw},
\begin{equation}\label{sw}
  \begin{split}
\ln  \E z^X
&=-\sumin\ln\Bigpar{1-\frac{1-z\qw}{p_i}}
\le 
-\sumin\frac{\px}{p_i}\ln\Bigpar{1-\frac{1-z\qw}{\px}}
\\&
= -\sumin\frac{\px}{p_i}\ln\Bigpar{1-\frac{\gl-1}{\gl-\px}}
= -\mu\px\ln\frac{1-\px}{\gl-\px}
= \mu\px\ln\frac{\gl-\px}{1-\px}.
  \end{split}
\end{equation}
Furthermore, by \eqref{z},
\begin{equation}
  \frac{1-z(1-\px)}{\px}
=
  \frac{1-(\gl-\px)/\gl}{\px} = \frac{1}{\gl}.
\end{equation}
Hence, \refL{L1}, \eqref{z} and \eqref{sw} yield
\begin{equation}\label{bx}
  \begin{split}
\ln	\P(X\ge \gl\mu)
&\le -\ln\gl -\gl\mu \ln z +\ln\E z^X
\\&
\le -\ln\gl -\gl\mu \ln \frac{\gl-\px}{\gl(1-\px)}
+\mu\px\ln\frac{\gl-\px}{1-\px}
\\&
=-\ln\gl+\gl\mu\ln (1-\px) + \mu f(\gl),
  \end{split}
\end{equation}
where
\begin{equation}
  \begin{split}
f(\gl)&:=
 -\gl \ln \frac{\gl-\px}{\gl}
+\px\ln\frac{\gl-\px}{1-\px}
\\&\phantom:
=
 -(\gl-\px) \ln \xpar{\gl-\px}
+\gl\ln\gl
-\px\ln\xpar{1-\px}
.
  \end{split}
\end{equation}
We have $f(1)=-\ln(1-\px)$ and, for $\gl\ge1$,
using \eqref{tp},
\begin{equation}\label{f'}
  \begin{split}
	f'(\gl)
=
- \ln \xpar{\gl-\px}
+\ln\gl
=-\ln\Bigpar{1-\frac{\px}{\gl}}
\le
-\frac{1}{\gl}\ln\xpar{1-{\px}}.
  \end{split}
\end{equation}
Consequently, by integrating \eqref{f'}, for all $\gl\ge1$,
\begin{equation}
  f(\gl)\le -\ln(1-\px) -\ln\gl\cdot\ln(1-\px),
\end{equation}
and the result \eqref{t2} follows by \eqref{bx}.
\end{proof}

\begin{remark}
  Note that for large $\gl$, the exponents above are roughly linear in
$\gl$, while for $\gl=1+o(1)$ we have $\gl-1-\ln\gl\sim\frac12(\gl-1)^2$ so
  the exponents are quadratic in $\gl-1$. The latter is to be expected from
  the central limit theorem. However, if $\gl=1+\eps$ with $\eps$ very small
  and the central limit theorem is applicable, then $\P(X\ge(1+\eps)\mu)$ is 
roughly $\exp(-\eps^2\mu^2/(2\gss))$, where $\gss=\Var X = \sumin \Var X_i
=\sumin\frac{1-p_i}{p_i^2}$. Hence, 
in this case 
the exponents in \eqref{xp} and \eqref{t2} are
asymptotically too small  by a factor of rougly,
for small $p_i$, 
\begin{equation}
\frac{\px\mu}{\mu^2/\gss}\approx \frac{\px\sumin p_i\qww}{\sumin p_i\qw},
\end{equation}
which may be much smaller than 1. (For example if $p_2=\dots=p_n$ and
$p_1=p_2/ n\qqq$.)
\end{remark}

\section{Upper bounds for the lower tail}

We can similarly bound the probability $\P(X\le\lex)$ for $\gl\le1$.
We give only a simple bound corresponding to \refT{T1}.
(Note that $\gl-1-\ln\gl>0$ for both $\gl\in(0,1)$ and $\gl\in(1,\infty)$.)

\begin{theorem}\label{TL1}
  For any $p_1,\dots,p_n\in(0,1]$ and any $\gl\le1$,
\begin{equation}\label{xp-}
  \P(X\le \lex)
\le 
 e^{-\px\mu(\gl-1-\ln \gl)}.
\end{equation}
\end{theorem}

\begin{proof}
We follow closely the proof of \refT{T1}.
If $t\ge0$, then by \eqref{pgf},
\begin{equation}
  \E e^{-t X_i} 
= \frac{p_i}{e^{t}-1+p_i}
\le  \frac{p_i}{t+p_i}
=\Bigpar{1+\frac{t}{p_i}}\qw.
\end{equation}
Hence
\begin{equation}
  \E e^{-t X} = \prodin \E e^{-t X_i} 
\le \prodin \Bigpar{1+\frac{t}{p_i}}\qw
\end{equation}
and, in analogy to \eqref{markov2},
\begin{equation}\label{bo-}
  \begin{split}
\P(X\le \gl \mu)\le e^{t\gl\mu} \E e^{-t X}
\le \exp \biggpar{t\gl\mu -\sumin\ln \Bigpar{1+\frac{t}{p_i}}}.
  \end{split}
\end{equation}
In analogy with \eqref{tpp}, still by the convexity of $-\ln x$,
\begin{equation}\label{tpp-}
-\ln \Bigpar{1+\frac{t}{p_i}}
 \le-\frac{\px}{p_i}\ln\Bigpar{1+\frac{t}{\px}},
\end{equation}
and \eqref{bo-} yields
\begin{equation}\label{box-}
  \begin{split}
\P(X\le \gl \mu)
&\le \exp\biggpar{t\gl\mu -\ln \Bigpar{1+\frac{t}{\px}} \sumin
  \frac{\px}{p_i}}
\\&
= \exp\biggpar{t\gl\mu -\px\mu\ln \Bigpar{1+\frac{t}{\px}}}.
  \end{split}
\end{equation}
%Choose $t=(1-\eps)\px$ for some $\eps\in(0,1)$.
Choosing $t=(\gl\qw-1)\px$, we obtain \eqref{xp-}. 
\end{proof}

\section{A lower bound}

We show also a general lower bound for the upper tail  probabilities,
which shows that for constant $\gl>1$,
the exponents in Theorems \ref{T1} and \ref{T2} are at most
a constant factor away from best possible.

\begin{theorem}
  \label{TL}
For any $p_1,\dots,p_n\in(0,1]$ and any $\gl\ge1$,
\begin{equation}\label{tl}
  \P(X\ge\lex)
\ge \frac{(1-\px)^{1+1/\px}}{2\px\mu}(1-\px)^{(\gl-1)\mu}.
\end{equation}
\end{theorem}

\begin{lemma}\label{LA}
  If $A\ge1$ and $0\le x\le 1/A$, then
  \begin{equation}
A\bigpar{x+\ln(1-x)}	
\le \ln\bigpar{1-Ax^2/2}.
  \end{equation}
\end{lemma}

\begin{proof}
  Let $f(x):=A\bigpar{x+\ln(1-x)}-\ln\bigpar{1-Ax^2/2}$.
Then $f(0)=0$ and
\begin{equation}
  f'(x)=A\Bigpar{1-\frac{1}{1-x}}+\frac{Ax}{1-Ax^2/2}
=
-\frac{Ax}{1-x}+\frac{Ax}{1-Ax^2/2}
\le0
\end{equation}
for $0\le x<1/A\le1$, since then $0<1-x\le 1-Ax^2/2$.
Hence $f(x)\le0$ for $0\le x\le 1/A$.
\end{proof}

\begin{proof}[Proof of \refT{TL}]
Let $\eps:=1/(\px\mu)$.
By \refT{TL1} (with $\gl=1-\eps$) and \refL{LA} (with $A=\px\mu\ge1$),
\begin{equation}
  \begin{split}
\P(X\le(1-\eps)\mu)
\le \exp\bigpar{-\px\mu(-\eps-\ln(1-\eps))}
\le 1-\frac{\px\mu\eps^2}2
=1-\frac{1}{2\px\mu}.
  \end{split}
\end{equation}
Hence,
$\P(X\ge(1-\eps)\mu) \ge \xqfrac{1}{2\px\mu}$, and by \refL{L0}(ii),
\begin{equation*}
  \P(X\ge\lex)
\ge (1-\px)^{(\gl-1+\eps)\mu+1}
\P(X\ge(1-\eps)\mu) 
\ge (1-\px)^{(\gl-1+\eps)\mu+1}\frac{1}{2\px\mu},
\end{equation*}
which completes the proof since $\eps\mu=1/\px$.
\end{proof}

\section{Exponential distributions}

In this section we assume that $X=\sumin X_i$ where $X_i$, $i=1,\dots,n$, are
independent random variables with exponential distributions:
$X_i\sim\Exp(a_i)$, with density function $a_ixe^{-a_i x}$, $x>0$,
and expectation $\E X_i=1/a_i$. (Thus $a_i$ can be interpreted as a rate.)
The exponential distribution is the continuous analogue of the geometric
distributions, and the results above have (simpler) analogues for
exponential distributions.
We now define
\begin{align}
  \mu&:=\E X = \sumin \E X_i = \sumin \frac{1}{a_i},
\\
\ax&:=\min_i a_i.
\end{align}

\begin{theorem}
  \label{Texp}
Let $X=\sumin X_i$ with $X_i\sim\Exp(a_i)$ independent.
\begin{romenumerate}[-10pt]
\item 
For any $\gl\ge1$,
\begin{equation}%\label{t2}
	\P(X\ge \lex)
\le\gl\qw e^{-\ax\mu(\gl-1-\ln\gl)}.
  \end{equation}
\item For any $\gl\ge1$, we have also the simpler but weaker
\begin{equation}%\label{c2}
  \P(X\ge \lex)
\le  e^{1-\gl} .
\end{equation}
\item 
For any $\gl\le1$,
\begin{equation}%\label{xp-}
  \P(X\le \lex)
\le 
 e^{-\ax\mu(\gl-1-\ln \gl)}.
\end{equation}
\item 
For any $\gl\ge1$,
\begin{equation}%\label{tl}
  \P(X\ge\lex)
\ge \frac{1}{2e\ax\mu}e^{-\ax\mu(\gl-1)}.
\end{equation}
\end{romenumerate}
\end{theorem}
\begin{proof}
  Let $X_i\NN\sim\Ge(a_i/N)$ be independent (for $N>\max_i a_i$).
Then $X_i\NN/N\dto X_i$, where $\dto$ denotes convergence in distribution, 
and thus $X\NN/N\dto X$, where
$X\NN:=\sumin X_i\NN$.
Furthermore, $\mu\NN:=\E X\NN=M\nu$ and $\px:=\min_i (a_i/N)=\ax/N$.
The results follow by taking the limit as \Ntoo{} in \eqref{t2}, \eqref{c2},
\eqref{xp-} and \eqref{tl}.
(Alternatively, we may imitate the proofs above, using $\E
e^{tX_i}=a_i/(a_i-t)$ for $t<a_i$.)
\end{proof}

\newcommand\AAP{\emph{Adv. Appl. Probab.} }
\newcommand\JAP{\emph{J. Appl. Probab.} }
\newcommand\JAMS{\emph{J. \AMS} }
\newcommand\MAMS{\emph{Memoirs \AMS} }
\newcommand\PAMS{\emph{Proc. \AMS} }
\newcommand\TAMS{\emph{Trans. \AMS} }
\newcommand\AnnMS{\emph{Ann. Math. Statist.} }
\newcommand\AnnPr{\emph{Ann. Probab.} }
\newcommand\CPC{\emph{Combin. Probab. Comput.} }
\newcommand\JMAA{\emph{J. Math. Anal. Appl.} }
\newcommand\RSA{\emph{Random Struct. Alg.} }
\newcommand\ZW{\emph{Z. Wahrsch. Verw. Gebiete} }
\newcommand\DMTCS{\jour{Discr. Math. Theor. Comput. Sci.} }

\newcommand\AMS{Amer. Math. Soc.}
\newcommand\Springer{Springer-Verlag}
\newcommand\Wiley{Wiley}

\newcommand\vol{\textbf}
\newcommand\jour{\emph}
\newcommand\book{\emph}
\newcommand\inbook{\emph}
\def\no#1#2,{\unskip#2, no. #1,} %(typeset after year) 
\newcommand\toappear{\unskip, to appear}

\newcommand\arxiv[1]{\url{arXiv:#1}}
\newcommand\arXiv{\arxiv}

\def\nobibitem#1\par{}


\begin{thebibliography}{99}

\bibitem{BLM}
St\'ephane Boucheron, 
G\'abor Lugosi
and Pascal  Massart, 
\emph{Concentration Inequalities},
Oxford Univ. Press, Oxford, 2013.


\bibitem{DemboZeitouni}  
Amir Dembo and Ofer Zeitouni,
\book{Large Deviations Techniques and Applications}.
 2nd ed.,
Springer, New York, 1998.


\bibitem[Janson, {\L}uczak and Ruci\'nski(2000)]{JLR}
Svante Janson, Tomasz \L uczak \& Andrzej Ruci\'nski,
\book{Random Graphs}.
\Wiley, New York, 2000.

\nobibitem{Gut}
A. Gut, 
\emph{Probability: A Graduate Course},
Springer, New York, 2005.

\bibitem{Kallenberg}
Olav Kallenberg,
\book{Foundations of Modern Probability.}
2nd ed., Springer, New York, 2002. 


\end{thebibliography}
\end{document}